\documentclass[]{amsart}

\usepackage{latexsym}
\usepackage{amssymb,amsmath,amsopn}
\usepackage[dvips]{graphicx}   %To insert ps figures
\usepackage{color,epsfig}      %To insert ps figures
\usepackage{url}
\usepackage{amsbsy}

\newtheorem{theorem}{Theorem}[]

\newtheorem{proposition}{Proposition}[]

\theoremstyle{definition}

\newtheorem{remark}[]{Remark}[]

\newtheorem{Pax}{Axiom}

\newtheorem*{statement*}{Statement}

\begin{document}
\title[]{The bridge between Desargues' and Pappus' theorems}
\author[\'A. G.Horv\'ath]{\'Akos G.Horv\'ath}
\address {ELKH-BME Morphodynamics Research Group and Department of Geometry, Institute of Mathematics,
Budapest University of Technology and Economics,
Műegyetem rkp. 3., H-1111 Budapest, Hungary}
\email{ghorvath@math.bme.hu}

\dedicatory{}
\subjclass{51A05,51A20,51A30}
\keywords{Desargues' theorem, Pappus' theorem, plane of incidence, projective plane}

\date{}

\begin{abstract}
In this paper, we investigate the configuration theorems of Desargues and Pappus in a synthetic geometric way. We provide a bridge between the two configurations with a third one that can be considered a specification for both. We do not use the theory of collineations or the analytic description of the plane over a ternary ring.
\end{abstract}

\maketitle

\section{Introduction}

The first geometric statement about projective planes appears to be that of Pappus, who postulated the collinearity of three points in a configuration of points and lines. His finding, the so-called Pappus' theorem initiated similar geometric studies, in which the basic objects are points and lines, and between them there is a fundamental relationship, which we call incidence. It is not clear whether the truth value of a statement about a configuration depends on what "space" the configuration is in. In order to answer this question, the basic properties of incidence must be fixed, so we need "axioms" that describe this relationship. The simplest such system is the axiom system of the plane of incidence (or projective plane). When we formulate additional true statements using the given axioms, we use the axiomatic approach of synthetic geometry. Extending this method to "higher dimensional" geometry requires more basic objects, the "planes" of the geometry, and more axioms between the new and old basic objects. This raises a new problem, namely whether the axiom system for the new space provides an opportunity to formulate stronger results for the planes in the space than the set of axioms for the plane or not. Desargues configuration and the corresponding theorem answer this question, in a well-defined projective space Desargues' theorem is always true, in contrast to the case of the plane of incidence, when it would be either true or false. Therefore, the axiomatic approach of synthetic plane geometry and the approach through the mutual position of planes and lines of spatial geometry give different results with respect to a specific plane. This leads to another approach to describing projective geometry: the algebraic (or analytic) approach to coordinate geometry. This study of projective planes is usually based on a coordinate method over an algebraic structure. In the first half of the twentieth century, many important questions were answered using this method, therefore the "modern" description of projective planes requires knowledge of the appropriate (sometimes very strange) algebraic structures.

In this paper, we focus only on the two basic configurations of projective geometry and give a configuration that connects them. Our method is synthetic and does not require algebraic knowledge. Starting from the spatial configuration of the real projective space, we arrived at the planar configuration with the above property through successive specializations. Since the vast literature on projective planes contains much information on such special configurations, we also systematically review the corresponding results.

\subsection{History}
Many works attempt to give an overview of this very special topic of geometry, but the terminologies used in most cases differ. In our bibliography we have listed some monographs \cite{baer, dembowski, hall-1959, hughes, seidenberg, stevenson, veblen-young}, as well as some reviews and research articles \cite{burn, grari, hall-1943, marchisotto, spencer, weibel, veblen-wedderburn}. On a plane of incidence, in which Desargues' theorem is true, we can assign homogeneous coordinates to the points, which are elements of a division ring (skew field). The skew field belonging to the Desarguesian plane is a (commutative) field if and only if Pappus' theorem is true in the plane. Consequently, Pappus' theorem is necessary to define the standard tools of "classical projective geometry". On the other hand, to create a well-functioning general coordinate system of the plane, it is sufficient to use the "Little-Desargues theorem", which is a weakened (or minor) version of Desargues' theorem. A plane with the Little-Desargues property is called a \emph{Moufang plane} (see \cite{moufang} ). (In this case, the coordinates form an alternative division ring.) If the Little-Desargues property is not fulfilled in all situations, then we need a coordinate system over a more general algebraic structure to define and examine the basic concepts of projective geometry. (Algebraically, the coordinates form a planar ternary ring, and the system is "local", related to four points in the plane that are located at general positions.)

Skornyakov's paper (see \cite{skornyakov}) contains the first detailed analysis of special (or minor) forms of Desargues and Pappus theorems. The notation of Skornyakov's paper covers the possible specializations of Desargues' configuration. Smaller forms of Desargues' theorem and Pappus' theorem were also investigated in some later papers, which enrich the literature with additional results (see, for example, \cite{pickert} and \cite{al-dhahir}).
When Skornyakov requires that one, two, or three vertices of one triangle fall on the lines of the corresponding side of the other triangle, he denotes the corresponding propositions $D_1, D_2$, and $D_3$. (For example, our proposition $wD$ is a proposition of type $D_1$ and the proof of Proposition \ref{thm:HDimpliesLD} contains configurations $D_1$, $D_2$ and $D_3$.) It also distinguishes two specializations when two extra fits is present in both, but in the first situation two sides of the first triangle contain the corresponding vertices of the second triangle, while in the other the first fit occurs on one side of the first triangle, while in the second it occurs on the corresponding side of the second triangle. While the first is denoted by $D_2$, the second by $D_1^1$. Finally, if the extra fit falls on a connecting line passing through the center of the perspective, the corresponding specialization is marked with a comma. (Our LD theorem is denoted by $D'$ and our wLD theorem by $D'_1$ in this system.) Similarly, in the article \cite{al-dhahir} we find five minor forms of the Pappus theorem, which are It is denoted by $P_i$, where $i = 1,2,3,4,5$. $P_2$ is the same as our LP theorem, $P_1$ is the dual of $P_2$ (the theorem we call Small-Pappus theorem), $P_3$ is called the Perspective-Pappus theorem, $P_4$ is its dual and $P_5$ is a statement about the incomplete Pappian chain of inscribed triangles. For the sake of clarity, we list the proven aspects of these works with their original designations.

\begin{itemize}\label{item:spa}
\item $D'$ is equivalent to $D_1$ (\cite{skornyakov}. \cite{pickert})
\item $D'_1$, $D_1^1$, $D_2$ are equivalent statements (\cite{al-dhahir}).
\item $D'_2$ is equivalent to $D_3$ (\cite{al-dhahir}).
\item $D'$ implies $P_1$ and $P_2$ (\cite{pickert}).
\item $D_2$ implies all $P_j$, where $j=1,2,3,4,5$. (\cite{al-dhahir})
\end{itemize}

In the rest of our article, we give the name of a specialization that refers to the derivation of the appropriate configuration.

\subsection{Results of the paper}

We prove four statements using a purely synthetic method, three of which were not found in the reviewed literature, and we provide a direct proof for the fourth known statement.

\begin{theorem}\label{thm:main}
The following statements are true for all planes of incidence.
\begin{itemize}
\item The Little-Desargues theorem follows from the homologous-Desargues theorem. (Proposition \ref{thm:HDimpliesLD}.)
\item The weak Little-Desargues theorem is equivalent to the strong Perspective-Pappus theorem. (Proposition \ref{thm:wLDequivalenttosPP}.)
\item The Little-Pappus theorem follows from the weak Little-Desargues theorem. (Proposition \ref{thm:wLDequivalenttoLP}.)
\end{itemize}
In the Euclidean model of the real projective plane, the Little-Desargues theorem can be proved from the weak Little-Desargues theorem. (\ref{thm:wLDimpliesLDinR} suggestion.)
\end{theorem}

The known implications between the special forms of Desargues' theorem and Pappus' theorem are shown in the diagram below. The black arrow denotes trivial or previously proven conclusions, the blue arrow denotes the well-known but synthetically verified conclusions in our article, and the red arrow those that are not found in the literature. (We do not consider the implications that are based on "duality" in the literature without direct evidence. This is because our investigation also applies to planes where the "dual" property does not exist.)

\begin{figure}[ht]
\centering
\includegraphics[width=\textwidth]{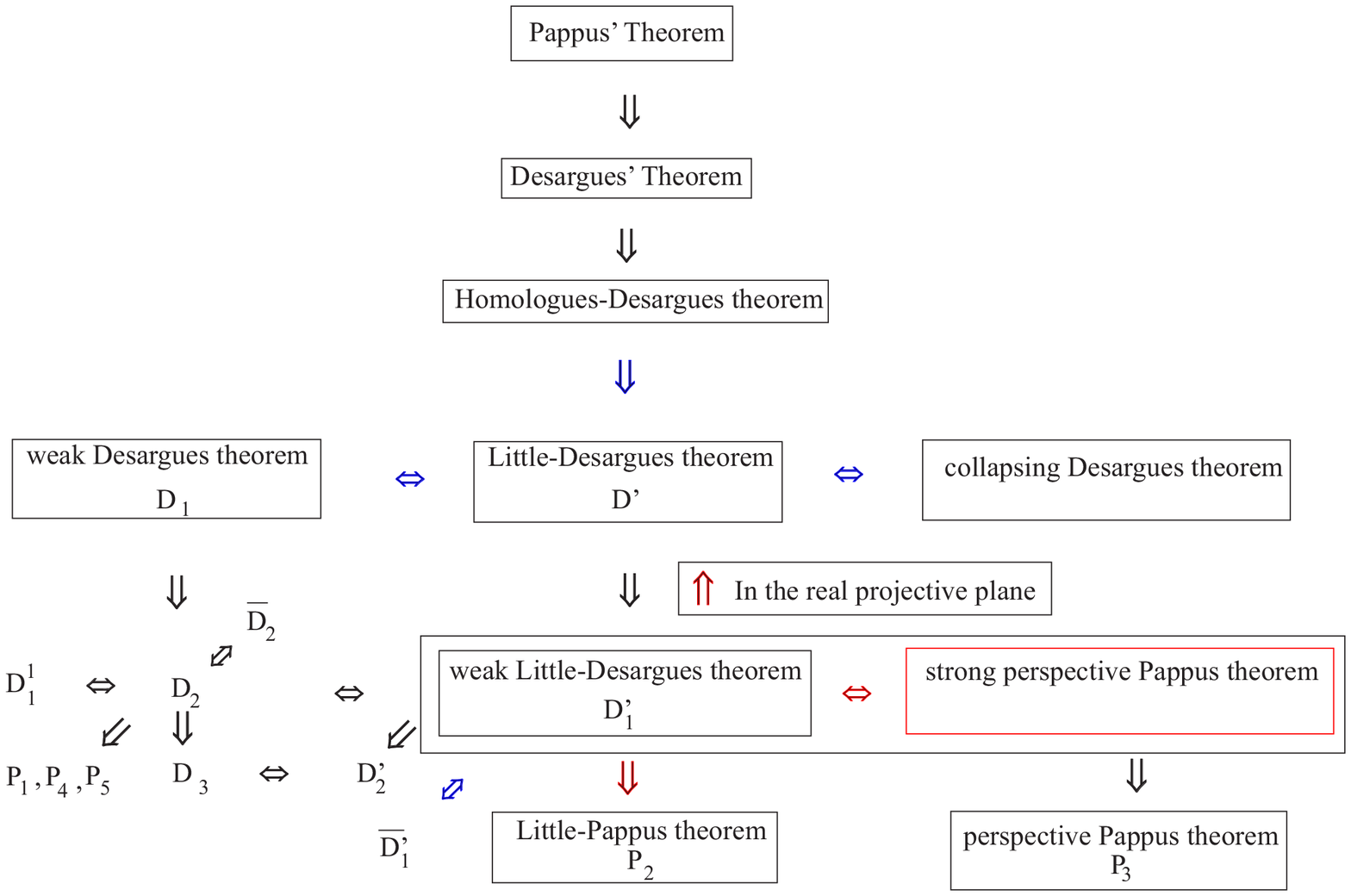}
\caption{The table of implications.}
\label{fig:implications}
\end{figure}

We have only two important questions left unanswered: \emph{Does the weak Little-Desargues theorem follow from the little-Pappus theorem? Does the Little-Desargues Theorem follow from the weak Little-Desargues Theorem?} For historical reasons, in this paper we denote many propositions by theorems, but prove only one theorem, which is the collection of the four "propositions" we have proved. The placement of these statements in the article corresponds to the logic of specialization.

\section{The plane of incidence and the basic configurations}

Veblen and Young give the \cite{veblen-young} synthetic structure of the incidence properties of $n$-dimensional projective geometry. They used two undefined objects, the point and the line, and the concepts of plane, space, and 4-space are defined based on them. In our treatment, the above "system of assumptions" changes to statements that are theorems in the Veblen-Young system. We lose the requirement of simplicity and independence, but we get a more didactic system of axioms, which leads to faster construction of the resulting planar geometry.

Some points are called collinear if they fall on the same line. The fitting of points and lines can be considered a so-called \emph{plane of incidence}, which satisfies the following three axioms:

\begin{Pax}
Two points uniquely determine a line which is incident with them.
\end{Pax}
\begin{Pax}
Two lines uniquely determine a point which is incident with them.
\end{Pax}
\begin{Pax}
There exists four points, no three of which are collinear.
\end{Pax}

The plane of incidence is a precursor to the concept of the classical projective plane. We can raise interesting problems in it, e.g. we can examine statements that assert the fit of certain lines and points based on the fit of other points and lines. Statements of this type are called \emph{configuration statements}. If a configuration theorem holds on the plane, then there is a \emph{configuration} containing exactly the incidences that the theorem requires. An example is Desargues' theorem:

\begin{theorem}[Desargues]\label{thm:desargues}
The lines incident with the pair of points $A,A'$; $B,B'$ and $C,C'$ are go through the same point $S$ if and only if the points $X$, $Y$ and $Z$ incident with the pairs of lines $AB$, $A'B'$; $BC$, $B'C'$ and $AC$, $A'C'$ are incident with the same line $s$.
\end{theorem}

\begin{figure}[ht]
\includegraphics[scale=1]{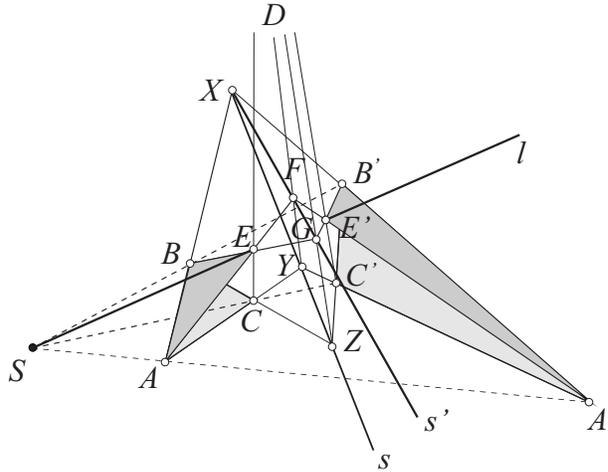}
\caption{The theorem of Desargues}
\label{fig:desargues}
\end{figure}

The question is: can we prove Desargues' theorem on a general plane of incidence? The answer is known, this theorem is usually not true. (A nice counterexample is given by Moulton in \cite{moulton}. For more information on non-Desarguesian planes, I recommend Weibel's \cite{weibel} survey.)

Note that the best-known example of the plane of incidence is the real projective plane for which this theorem holds. It is also immediately noticeable that the opposite statement also follows from the proved direction if we apply the statement to the pairs of points $AB$, $A'B'$ and $YZ$ of Figure \ref{fig:desargues} which are perspective from the point $X$. In real projective geometry, a proof can also be obtained by using space axioms. The figure \ref{fig:desargues} contains a complete $\{S,A,A',B,B',E,E',X,F,G\}$ Desargues configuration embedded in the three-dimensional projective space, its central projection onto the plane $SAB$ gives the plane Desargues configuration of ten points $\{S,A,A',B,B',C,C',X,Y,Z\}$.

Another important statement of real projective geometry is the Pappus theorem. It is not valid on all planes of incidence, not even on all Desarguesian planes. Hessenberg proved in \cite{hessenberg} that in a plane of incidence Pappus' theorem implies Desargues' theorem. This means that a plane satisfying Pappus' theorem is always Desarguesian. The Pappus configuration is the complete configuration of the Pappus' theorem:

\begin{theorem}[Pappus]\label{thm:pp}
Assume that $A,B,C$ are three points of a line and $A',B',C'$ are three points of an other line. Then the points $C''=AB'\cap A'B$, $B''=AC'\cap A'C$ and $A''=BC'\cap B'C$ are collinear.
\end{theorem}

Based on the figure \ref{fig:ppandpb}, we can easily verify that the dual form of this theorem (which was first proved by Brianchon) is equivalent to the original statement. We have:

\begin{figure}[ht]
  \centering
    \includegraphics[scale=0.8]{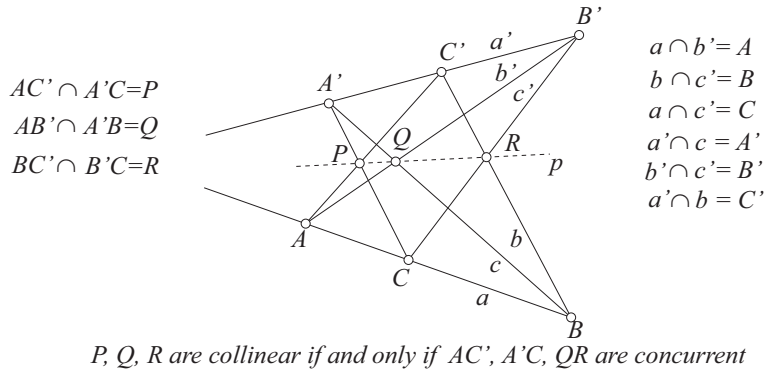}\\
  \caption{The equivalence of the Pappus's theorem and the theorem of Brianchon}\label{fig:ppandpb}
\end{figure}

\begin{theorem}[Pappus-Brianchon]\label{thm:pb}
Assume that $a,b,c$ are three lines through a point and  $a',b',c'$ are three lines through another point. Then the lines $c''=(a\cap b',a'\cap b)$, $b''=(a\cap c',a'\cap c)$ and $a''=(b\cap c',b'\cap c)$ are concurrent.
\end{theorem}

The complete Pappus configuration contains nine points and nine lines, with three fits on each line and each point. In real projective geometry, there is a nice connection with Gallucci's theorem on the transversals of skew lines. This connection can be seen in the figure \ref{fig:gpp} (see \cite{gho} for more details).

\begin{figure}[ht]
  \centering
    \includegraphics[scale=0.5]{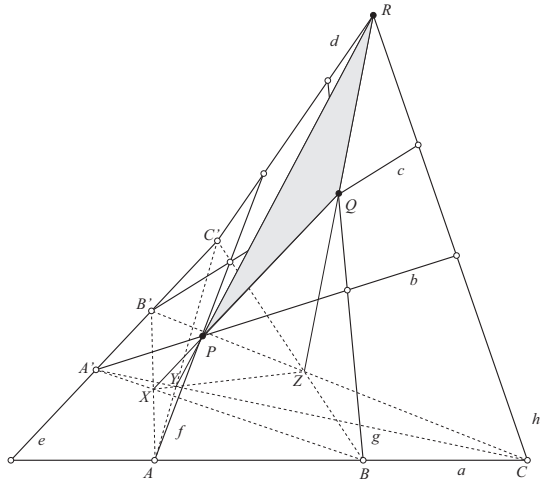}\\
  \caption{Gallucci's theorem is equivalent to Pappus' theorem}\label{fig:gpp}
\end{figure}

\section{Specializations of Desargues' theorem and the corresponding configurations}

In the figure \ref{fig:desargues}, we see that the planar configuration $\{S,A,A',B,B',E,E',X,F,G\}$ can be specialized by the choice of the center of the projection. For example, if we choose $D$ on the plane of the line $s'$ and the point $S$, we get that the plane configuration will contain a new incidence relation, namely $S\in s$.

\begin{figure}[ht]
  \centering
    \includegraphics[scale=1]{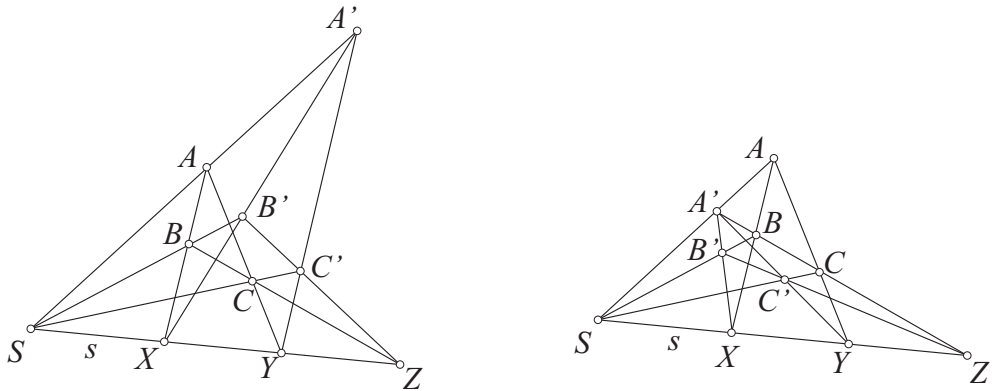}\\
  \caption{Little-Desargues theorem and the weak Little-Desargues theorem}\label{fig:littledesargues}
\end{figure}

This new configuration (see the left side of the figure \ref{fig:littledesargues}) is the configuration of the Little-Desargues theorem. In this case, the statement is:

\begin{theorem}[LD theorem]\label{thm:LD}
The lines incident with the pair of points $A,A'$; $B,B'$ and $C,C'$ are go through the same point $S$, moreover $S,X=AB\cap A'B', Y=AC\cap A'C'$ are also collinear then $Z=AB\cap A'B'$ is incident with the line $s=XY$.
\end{theorem}

The configuration of the Little-Desargues theorem is not as symmetric as the Desargues configuration. There are four incident lines (points) on a point (a line), and the rest match only three lines (points). It is clear that the Little-Desargues Theorem is a logical consequence of the Desargues Theorem, but it is not clear whether the Little-Desargues Theorem follows from a restriction of the Desargues Theorem in which we also require that the perspective axis $s$ does not contain the center $S$ of the perspectivity. This last condition is the condition of the so-called homologous-Desargues theorem:

\begin{theorem}[HD theorem]\label{thm:HD}
The lines incident with the pair of points $A,A'$; $B,B'$ and $C,C'$ are go through the same point $S$ and $S,X=AB\cap A'B', Y=AC\cap A'C'$ aren't collinear then $Z=BC\cap B'C'$ is incident with the line $s=XY$.
\end{theorem}

The theory of collineations gives a positive answer to the former question, as it can be read in the book \cite{stevenson} (see Theorem 5.2.5). We now provide an immediate proof of this fact.

\begin{proposition}\label{thm:HDimpliesLD}
The Homologue-Desargues  theorem implies the Little-Desargues theorem.
\end{proposition}

\begin{proof}
Consider the notation of Figure \ref{fig:HDimpliesLD}.

\begin{figure}[ht]
  \centering
    \includegraphics[scale=1]{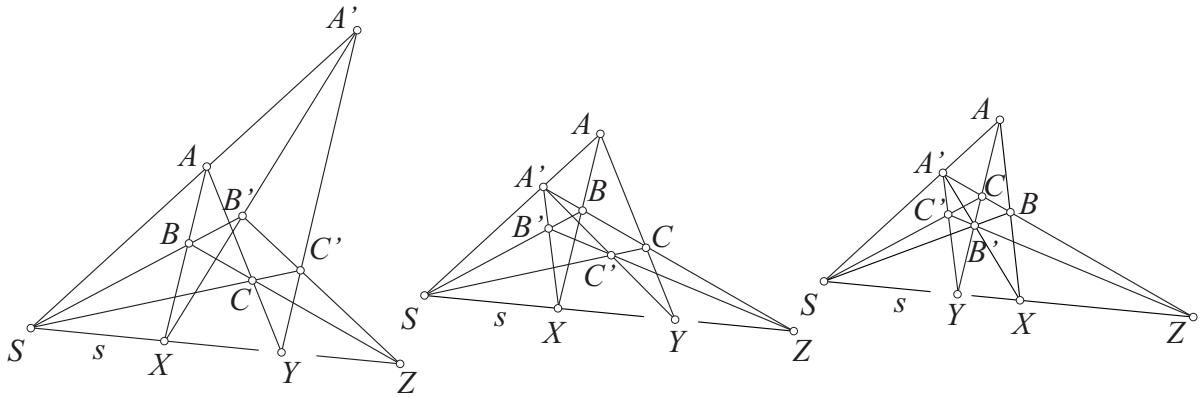}\\
  \caption{Homologue-Desargues theorem implies Little-Desargues theorem}
  \label{fig:HDimpliesLD}
\end{figure}

We need to prove that $Y$ is incident with $s=XZ$. Consider the triangles $AXY$ and $SB'C'$, the center of perspectivity is $A'$. (We use the notation $AXY \overset{A'}{\sim} SB'C'$ for these perspective triangles.) Since $AX\cap SB'=B$ and $AY\cap SC'=C$ there are two possibilities. If $A'$ is not incident with the line $BC$, then we can apply the HD theorem and obtain that the point $XY\cap B'C'$ is incident with the line $BC$, that is, the line that $X$ and $Z=BC\cap B'C'$ is determined incident to $Y$, as we said.

If $A'$ is incident to $BC$, then we cannot apply the HD theorem, but we get a more specialized figure (see the middle image of Figure \ref{fig:HDimpliesLD}). Consider the perspective triangles $CZY \overset{C'}{\sim} SB'A'$. Then $CZ\cap SB'=B$ and $CY\cap SA'=A$ and again we have two options. If $AB$ does not contain $C'$, then the HD theorem means that the point $ZY\cap B'A'$ lies on the line $AB$, so the points $X,Z$ and $ Y$ are collinear. .

In the other case, $C'$ lies on the straight line $AB$ (see the right image in the same figure). Let us now consider the perspective triangles $BXZ \overset{B'}{\sim} SA'C'$. Since $BX\cap SA'=A$ and $BZ\cap SC'=C$ , the two possibilities are that $B'$ is not on $AC$ , or $B'$ is incident on with $AC$. In the first case, applying the HD theorem, we get that $XZ\cap A'C'$ lies on $AC$, and our statement is fulfilled again.

Finally, suppose that $A,C$ and $B'$ are collinear point triples. Then we have $A'XY \overset{A}{\sim} SBC$, which means that $A'X\cap SB=B'$, $A'Y\cap SC=C'$, and $XY\ cap BC$ are collinear if we can apply the HD theorem. If this is not the case, then $A,B'$ and $C'$ are collinear, and therefore $AB'=B'C'$. Since $B'$ is incident with $AC$, then $C$ lies on $AB'=B'C'$, therefore $A,B',C,C'$ are collinear points. Similarly, $C'$ lies on $AB$, so $AC'$ also contains $B$. Thus, $A,B',C,C',B$ are collinear points, therefore all points lie on the same line, and the statement is fulfilled trivially.
\end{proof}

\begin{remark}\label{rem:moufang}
The reverse of the statement is not true. A plane of incidence is called  \emph{Moufang plane} if it can be coordinatized over an alternative division ring. In the Moufang plane, the Little-Desargues theorem holds, but the validity of the Homologue-Desargues theorem depends on whether the coordinate ring is associative or not. Since the Cayley plane over the octonion is a plane of incidence over a non-associative alternative division ring, this is an example of a non-Desarguesian Moufang plane. In other words, in general, the LD theorem doesn't implies the HD theorem.
\end{remark}

If we choose the point $D$ on the intersection of the planes $SFG$ and $BGA'$ in Figure \ref{fig:desargues}, the resulting planar configuration is further specialized, as shown in the right image of Figure \ref{fig:littledesargues}. This configuration induced the weak Little-Desargues theorem (wLD theorem).

\begin{theorem}[wLD theorem]\label{thm:wLD}
The lines incident with the pair of points $A,A'$; $B,B'$ and $C,C'$ are go through the same point $S$, moreover the points $S,X=AB\cap A'B', Y=AC\cap A'C'$, and $A',B,C$ are also collinear then $Z=BC\cap B'C'$ is incident with the line $s=XY$.
\end{theorem}

The configuration of the wLD theorem contains two points (two lines) with four matching lines (points), and the other points (lines) have three incidences with the lines (points) of the configuration.

\begin{remark} Note that in \ref{fig:desargues} there is only one possibility to choose the centre point $D$ for which the points $A,C$ and $B'$ are collinear when $D= SFG\ cap BGA' \ cap AFB'$. With this choice, we get a planar configuration similar to the image on the right of the figure \ref{fig:HDimpliesLD}. It is clear that further specialization of the planar configuration is impossible with this method, because the position of the line $AB$ and the point $C'$ is independent of the choice of $D$.
\end{remark}

An important observation is that the logical consequence of the wLD theorem is its converse, i.e. the following statement is true:

\begin{theorem}[cwLD theorem]\label{thm:cwLD}
Assume that the points $S=AA'\cap BB'$,$X=AB\cap A'B', Y=AC\cap A'C'$, $Z=BC\cap B'C'$ are collinear and the points $A',B,C$ are also collinear then $S$ incident with the line $CC'$.
\end{theorem}

\begin{proof}
Consider the middle image in Fig. \ref{fig:HDimpliesLD}. Let $C^\star$ be the intersection of $SC$ and $A'C'$ and let $Z^\star:=B'C^\star \cap BC$ . Then $ABC \overset{S}{\sim} A'B'C^\star$, $AC\cap A'C^\star=Y$, $AB\cap A'B'=X$, where $ S,X,Y$ are collinear (and so are $A',B,C$). Applying the wLD theorem, we get that the points $BC\cap B'C^\star=R^\star$ $X,Z$ are collinear, which can only happen if $R=R^\star$ and thus $ C'=C^\star$. Therefore, the triplet $S,C,C'$ is also collinear.
\end{proof}

We do not know whether the LD theorem follows from the wLD theorem or not, but such a statement can be verified in a space with a metric.

\begin{proposition}\label{thm:wLDimpliesLDinR}
In the Euclidean model of the real projective plane, the Little-Desargues theorem can be proved by using only the weak Little-Desargues theorem.
\end{proposition}

\begin{figure}[ht]
  \centering
    \includegraphics[scale=1]{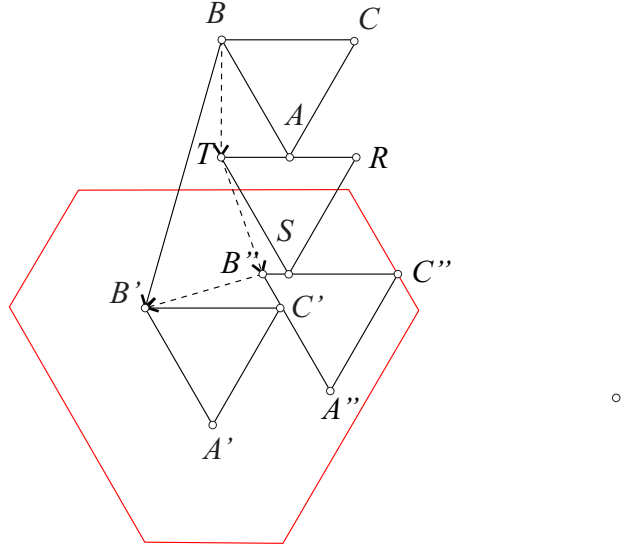}\\
  \caption{In the real projective plane the weak Little-Desargues theorem implies the Little-Desargues theorem}\label{fig:wLDimpliesLD}
\end{figure}

\begin{proof}
Suppose that the axis of perspectivity is the line $z$ at infinity. In this case, two triangles are axially perspective with respect to $z$ if their corresponding sides are parallel to each other. Two triangles satisfy the conditions of the LD theorem if they are parallel translates of each other, and satisfy the conditions of the wLD theorem if they are translates of each other, and one of the vertices of the first triangle falls on the side line of the other triangle that connects the vertices corresponding to the other two vertices of the first triangle. To prove the statement, we give a chain of triangles $ABC=A_0B_0C_0\cong ... \cong A_nB_nC_n=STR\cong A''B''C''$ which holds the wLD conditions such that the last triangle $A''B''C''$ also contains the wLD conditions for $A'B'C'$. Since Archimedes' axiom applies to the Euclidean plane, this process can be done with a finite number of triangles. In fact, consider a chain of translated copies of triangles $ABC=A_0B_0C_0\cong...\cong A_nB_nC_n=STR$, where for two consecutive triangles, one of the vertices of the second triangle fits on one of the sides of the one before it, and one of the vertices of the last triangle (say $A_n=S$) is in the closed domain consisting of the points of the translated copies intersecting the triangle $A'B'C'$. Define the triangle $A''B''C''$ to be the translation of $ABC$ with the property that its one side contains the point $S$ and one of its sides contains a vertex of the triangle $A'B'C'$. Since the parallelism is transitive, with the notations in the figure \ref{fig:wLDimpliesLD}, we get that $A''B''$, $A'B'$ ; and $A''C''$, $A'C'$ are parallel pair of segments, so $A'A''$, $B'B''$ and $C'C''$ are also parallel. Since $BC$ is parallel to $B''C''$, we get that $BC$ is parallel to $B'C'$ if and only if $B''C'' $ is parallel to $B'C'$. Since the configuration corresponding to the triangles $A'B'C'$ and $A''B''C''$ is a wLD configuration, this means that the LD theorem follows from the wLD theorem, as stated.
\end{proof}

In the figure \ref{fig:desargues}, we can choose point $D$ either on the $BGA'$ plane (but usually not on the $SFG$ plane) or on the $AA'G$ plane, to define new special cases of Desargues' theorem. The first planar configuration leads to the \emph{weak Desargues theorem } (wD theorem) and the second to the \emph{collapsing Desargues theorem} (cD theorem).

\begin{theorem}[wD theorem]\label{thm:wD}
The lines incident with the pair of points $A,A'$; $B,B'$ and $C,C'$ are go through the same point $S$, moreover the points  $A',B,C$ are collinear then $ X=AB\cap A'B', Y=AC\cap A'C'$, and $Z=BC\cap B'C'$ are also collinear.
\end{theorem}

\begin{theorem}[cD theorem]\label{thm:cD}
The lines incident with the pair of points $A,A'$; $B,B'$ and $C,C'$ are go through the same point $S$, let be $ X=AB\cap A'B', Y=AC\cap A'C'$, and $Z=BC\cap B'C'$ if moreover the points  $A,A',Z$ are collinear then $X,Y$ and $Z$ are also collinear.
\end{theorem}

Exercise 8 in \cite{stevenson} on page 203 states that the theorems LD, wD, and cD are equivalent to each other. Direct proofs of these equivalences can be reconstructed based on Figure \ref{fig:wDcD}.

\begin{figure}[ht]
  \centering
    \includegraphics[scale=1]{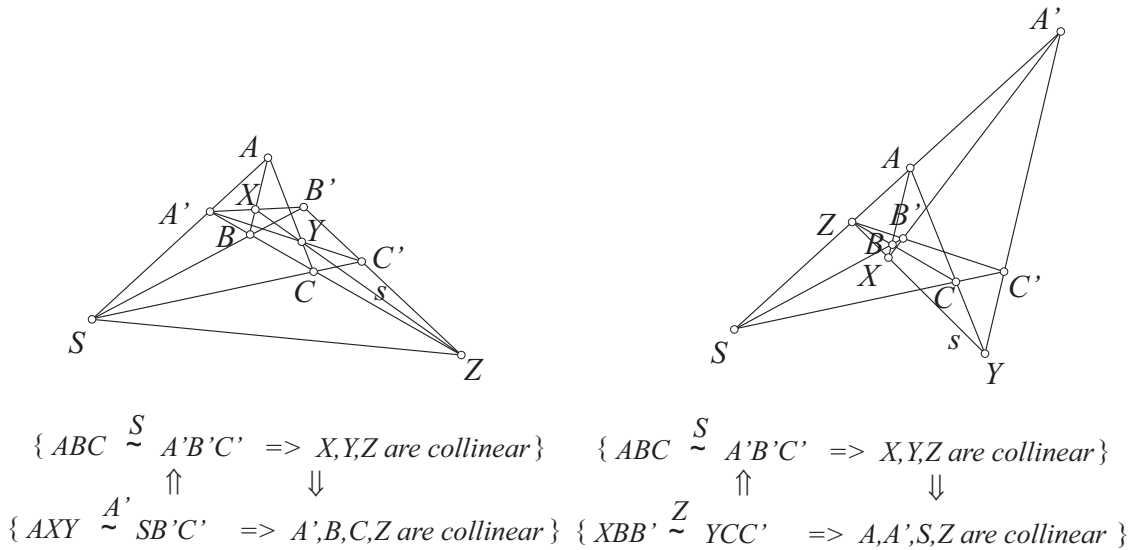}\\
  \caption{The three types of specialization are equivalent to each other.}\label{fig:wDcD}
\end{figure}

\section{Specializations of Pappus' theorem and the corresponding configurations}

Our first observation is that the configuration of the wLD theorem can also be considered a specialized configuration of Pappus' theorem. If we rewrite the notation of the original left-hand configuration in the figure \ref{fig:wLDandSPP} and change the position of the dashed line that indicates the conclusion in the theorem, we can formulate a Pappus-type theorem whose configuration is the right-hand image of the figure. Therefore, this configuration can be considered a "bridge" between Desargues' theorem and Pappus' theorem.

\begin{figure}[ht]
  \centering
    \includegraphics[scale=1]{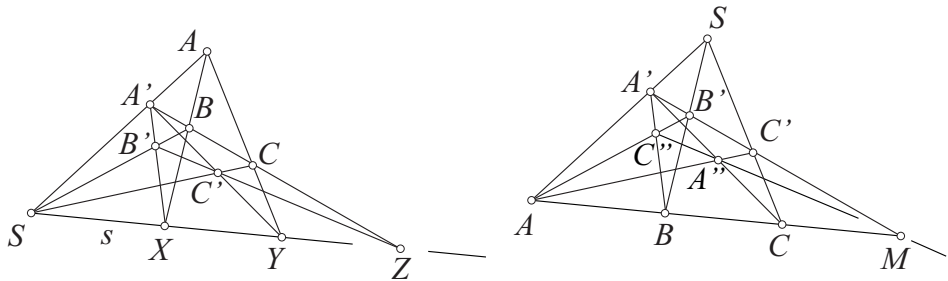}\\
  \caption{The "bridge configuration" between the Desargues' theorem and Pappus' theorem.}\label{fig:wLDandSPP}
\end{figure}

\begin{theorem}[sPP theorem]\label{thm:sPP}
Assume that $A,B,C$ are three points of one line and $A',B',C'$ are three points of another line, where $M$ is a common point of the two lines. Assume that the triples $S,A,A'$ $S,B,B'$ and $S,C,C'$ are also collinear. Then the points $M$, $C''=AB'\cap A'B$ and $A''=BC'\cap B'C$ are collinear.
\end{theorem}

We call this theorem \emph{strong perspective Pappus theorem} (spP theorem) because its logical consequence is the \emph{perspective Pappus theorem}:

\begin{theorem}[pP theorem]\label{thm:pP}
Assume that $A,B,C$ are three points of one line and $A',B',C'$ are three points of another line. Assume that the triples $S,A,A'$ $S,B,B'$ and  $S,C,C'$ are also collinear. Then the points $C''=AB'\cap A'B$, $B''=AC'\cap A'c$ and $A''=BC'\cap B'C$ are collinear.
\end{theorem}

This theorem is called "Specialized Pappus Theorem One" by \cite{stevenson} (see Exercise 3 on page 212) and is true on the Desarguesian plane.
Notice that in the Moulton plane it is easy to specify the position of the points $S,A,B,C,A',B',C'$ in such a way that the assertion of the pP theorem is fulfilled, but the points $A'',B '',C''$ and $M$ are not collinear. This shows that the sPP theorem is stronger than the PP theorem. The following theorem proves the legitimacy of the name "bridge configuration".

\begin{proposition}\label{thm:wLDequivalenttosPP}
The weak Little-Desargues  theorem is equivalent to the strong perspective Pappus theorem.
\end{proposition}

\begin{figure}[ht]
  \centering
    \includegraphics[scale=1]{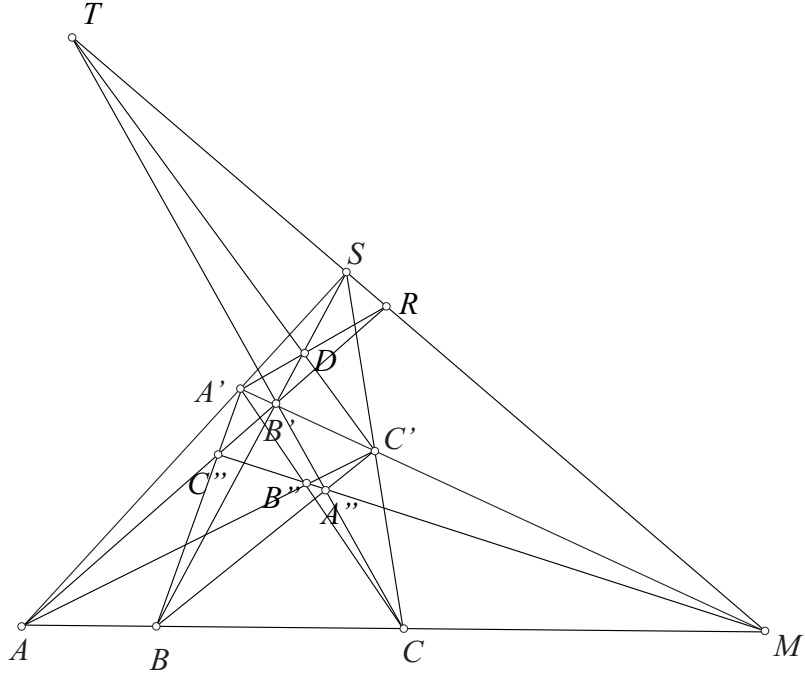}\\
  \caption{Configuration for proofs}\label{fig:wLDeqsPP}
\end{figure}

\begin{proof}
Considering the notations in the figure \ref{fig:wLDeqsPP}, we see that $C''B'A''\overset{B}{\sim} A'SC'$ perspective is satisfied. Since $B'A''\cap SC'=C$, $B'C''\cap SA'=A$, $A,B,C$ and $A',B',C'$ are collinear , we can apply the wLD theorem and obtain that $A'C' \cap A''C''$, $A$ and $C$ are also collinear. It follows that the assertion of the spP theorem is fulfilled.

Conversely, suppose that the sPP theorem holds on the plane. Let $A'DC'$ and $AB'C$ be two given triangles with perspective $A'DC'\overset{S}{\sim} AB'C$, on which the points $A',B',C'$ is collinear, and $R:=AB'\cap A'D$, $T=C'D\cap CB'$ and $S$ are also collinear. Then $TSR \overset{D}{\sim} C'B'A'$, therefore the collinear triples $T,S,R$ and $C',B',A'$ are in perspective position and are satisfied , that $TB'\cap C'S =C$ and $B'R\cap SA'=A$. It follows from the sPP theorem that the line $AC$ contains the intersection of the lines $TR$ and $A'C'$, which means that the points $R,S,T$ and $AC\cap A'C'= M$ is collinear. That is, the statement of the wLD theorem is true for the examined triangles.
\end{proof}

There is another specialization of Pappus' theorem, whose extra condition is analogous to the condition of the LD theorem, which we call \emph{Little-Pappus theorem} rather than "Specialized Pappus theorem Two", as \cite{stevenson} .

\begin{theorem}[LP theorem]\label{thm:LP}
Assume that $A,B,C$ are three points of a line and $A',B',C'$ are three points of another line whose point of intersection is $M$. Assume that the points $C''=AB'\cap A'B$, $A''=BC'\cap B'C$ and $M$ are collinear. Then $C''=AB'\cap A'B$, $B''=AC'\cap A'c$ and $A''=BC'\cap B'C$ are collinear, too.
\end{theorem}

As shown in Exercise 4 on page 212 of \cite{stevenson}, this theorem follows from Desargues' theorem. Now we prove a little more:

\begin{proposition}\label{thm:wLDequivalenttoLP}
The weak Little-Desargues  theorem implies the Little-Pappus theorem.
\end{proposition}

\begin{proof}
To prove the statement, consider the figure \ref{fig:wLDeqsPP} and assume that $A'',C''$ and $M$ are collinear. Let $S=AA'\cap CC'$, $R=AB'\cap SM$, $D=B'S\cap A'R$, and $T=C'D\cap SM$. Then $AB'C\overset{S}{\sim} A'DC'$ holds, where $AB'\cap A'D=R$, $A'C'\cap AC=M$. Now $A',B',C'$; Point triples $R,S,M$ are collinear, therefore, according to the wLD theorem, $DC'\cap CB'$ lies on the line $SM$, which means that $C,B',T$ is collinear. Now consider the triangles $SAC$ and $BC''A''$, with their vertices $SA\cap BC''=A'$, $SC\cap BA''=C'$ $AC\cap C''A''=M$ is triple collinear and $AC''\cap CA''=B'$ is on their line. Since $B,A,C$ are also collinear, we get from the cwLD theorem that $B'$ is a point on the $SB$ line. This means that in the case of $A,B,C$ and $A',B',C'$ the assumption of the pP theorem holds, which is a consequence of the spP theorem, therefore $A'',B''$ and $C''$ are also collinear.
\end{proof}

\end{document}